\documentclass{article}
\usepackage{amsmath}
\usepackage{amssymb}
\newtheorem{theorem}{Theorem}

\newtheorem{example}[theorem]{Example}

\newtheorem{lemma}[theorem]{Lemma}

\newtheorem{proposition}[theorem]{Proposition}
\newtheorem{remark}[theorem]{Remark}

\newenvironment{proof}[1][Proof]{\noindent\textbf{#1.} }{\ \rule{0.5em}{0.5em}}

\renewcommand{\nolimits}{}
\renewcommand{\limits}{}

\def\tto{\;{\lower 1pt \hbox{$\rightarrow$}}\kern -10pt
\hbox{\raise 2pt \hbox{$\rightarrow$}}\;}

\def\N{I\!\!N}

\def \N{I\!\!N}

\newcounter{lk}

\begin{document}

\begin{center}
\textbf{QUANTITATIVE STABILITY OF LINEAR INFINITE INEQUALITY
SYSTEMS UNDER BLOCK PERTURBATIONS WITH APPLICATIONS TO CONVEX
SYSTEMS}\footnote{This research was partially supported by grants
MTM2008-06695-C03 (01-02)
from MICINN (Spain).}\\[3ex]
M. J. C\'{A}NOVAS\footnote{Center of Operations Research, Miguel
Hern\'{a}ndez University of Elche, 03202 Elche (Alicante), Spain
(canovas@umh.es, parra@umh.es).}, M. A.
L\'{O}PEZ\footnote{Department of Statistics and Operations
Research, University of Alicante, 03080 Alicante, Spain
(marco.antonio@ua.es).}, B. S. MORDUKHOVICH\footnote{Department of
Mathematics, Wayne State University, Detroit, MI 48202, USA
(boris@math.wayne.edu). The research of this author was partially
supported the US National Science Foundation under grants
DMS-0603848 and DMS-1007132.} and \linebreak J.
PARRA\footnotemark[2]
\end{center}

{\small \textbf{Abstract.} The original motivation for this paper
was to provide an efficient quantitative analysis of convex
infinite (or semi-infinite) inequality systems whose decision
variables run over general infinite-dimensional (resp.\
finite-dimensional) Banach spaces and that are indexed by an
arbitrary fixed set $J$. Parameter perturbations on the right-hand
side of the inequalities are required to be merely bounded, and
thus the natural parameter space is $l_{\infty }(J)$. Our basic
strategy consists of linearizing the parameterized convex system
via splitting convex inequalities into linear ones by using the
Fenchel-Legendre conjugate. This approach yields that arbitrary
bounded right-hand side perturbations of the convex system turn on
constant-by-blocks perturbations in the linearized system. Based
on advanced variational analysis, we derive a precise formula for
computing the exact Lipschitzian bound of the feasible solution
map of block-perturbed linear systems, which involves only the
system's data, and then show that this exact bound agrees with the
coderivative norm of the aforementioned mapping. In this way we
extend to the convex setting the results of \cite{CLMP09}
developed for arbitrary perturbations with no block structure in
the linear framework under the boundedness assumption on the
system's coefficients. The latter boundedness assumption is
removed in this paper when the decision space is reflexive. The
last section provides the aimed application to the convex
case.}\vspace*{ 0.05in}

{\small \textbf{Key words.} semi-infinite and infinite programming,
parametric optimization, variational analysis, convex infinite inequality
systems, quantitative stability, Lipschitzian bounds, generalized
differentiation, coderivatives, block perturbations \vspace*{0.05in} }%
\vspace*{0.05in}

\bigskip

\section{Introduction}

This paper arose motivated by the extension to convex inequality
systems of some results from \cite{CLMP09} concerning
quantitative/Lipschitz stability of feasible solutions to linear
infinite and semi-infinite systems. The basic idea was to use the
so-called \emph{standard linearization }by means of the \emph{
Fenchel-Legendre conjugate. }This linearization approach entails
that each convex inequality is split into a generally infinite
system of linear inequalities; so that a right-hand side
perturbation of each convex inequality yields the same
perturbation for all the linear inequalities coming from splitting
the convex one. In this way, we are dealing with a linear
inequality system subject to {\em block perturbations}. Based on
this initial motivation we firstly analyze in a general framework
the Lipschitz stability of linear systems under \emph{arbitrary}
block perturbations.

Indeed, the methodology of block perturbations for linear systems
and their applications to convex inequalities has been previously
developed in \cite{CLPT10} to compute the distance to
ill-posedness for such systems, although now the parameter spaces
associated with block partitions are different from those in
\cite{CLPT10}. Going a bit further back, extreme cases of constant
perturbations are implicitly present along some proofs in
\cite{CDLP05,CLPT05}. This observation on the prominent role of
constant perturbations is also pointed out in the very recent
preprint \cite{Ioffe2010}  that provides an alternative
methodology to approach directly convex systems, where the concept
of perfect regularity plays a central role.

The expression obtained in the present paper for the \emph{exact
Lipschitzian bound }(also called Lipschitz modulus; see the
definition below) of the feasible set mapping provides a natural
extension of its linear counterpart \cite[Theorem~4.6]{CLMP09};
cf.\ also \cite[Corollary 3.2]{CDLP05} and \cite[Theorem
1]{CGP08}). In this sense, the methodology and proofs themselves
can be treated as major contributions of this paper. Specifically
we emphasize, aside from the methodology, the usage of tools such
as \emph{coderivatives} and the \emph{extended Ascoli formula} of
Lemma~\ref{Lem_distance}.

Consider the linear inequality system {\normalsize
\begin{equation}
\big\{\left\langle a_{t}^{\ast },x\right\rangle \leq b_{t},\ t\in
T\big\} \label{nominal sys}
\end{equation}
}referred to as the{\normalsize \ }\emph{nominal system}, where
$T$ is an arbitrary \emph{index set}, $x\in X$ is a \emph{decision
variable} from a general Banach space $X$ with its topological
dual $X^{\ast }$, and where the function $T\ni t\mapsto \left(
a_{t}^{\ast},b_{t}\right)\in X^{\ast }\times \mathbb{R}$ providing
the nominal system's data is also arbitrary. When $T$ is infinite
and $X$ is finite-dimensional, we are dealing with
\emph{semi-infinite }systems whereas \emph{infinite }systems allow
for both infinitely many inequalities and infinite-dimensional
decision spaces. Our approach involves considering a partition of
the index set $T$ denoted by
\begin{equation*}
\mathcal{J}:=\left\{ T_{j}\mid j\in J\right\},
\end{equation*}
i.e., $T_{j}\ne\varnothing $ for all $j\in J$ and
\begin{equation*}
T=\bigcup_{j\in J}T_{j}\text{ with }T_{i}\cap T_{j}=\varnothing\;
\text{ if }\; i\neq j.
\end{equation*}%
In the sequel the sets $T_{j},$ $j\in J,$ in the partition are
referred to as \emph{blocks}. Then we consider the parameterized
system
\begin{equation}
\sigma _{\mathcal{J}}\left( p\right) :=\big\{\left\langle
a_{t}^{\ast },x\right\rangle\le b_{t}+p_{j},\ t\in T_{j},~j\in
J\big\}, \label{linear sys j}
\end{equation}
where the \emph{perturbation parameter }$p=\left( p_{j}\right) _{j\in J}$
ranges on the Banach space $l_{\infty }(J)$ endowed with the norm
\begin{equation*}
\left\Vert p\right\Vert :=\sup_{j\in J}\left\vert p_{j}\right\vert .
\end{equation*}
The zero function $\overline{p}=0$ is regarded as the
\emph{nominal parameter}, which corresponds to the nominal system
$($\ref{nominal sys}$),$\ which coincides with $\sigma
_{\mathcal{J}}\left( 0\right) \ $for every partition
$\mathcal{J}$. From now on, in order to simplify the notation, the
nominal system $($\ref{nominal sys}$)$ is denoted just by $\sigma
\left(0\right) $. The two extreme partitions are
\begin{equation}
\mathcal{J}_{\min }:=\left\{ T\right\} \text{ and
}\mathcal{J}_{\max }:=\big\{\left\{t\right\}\big|\;\;t\in
T\big\}\label{J min max}
\end{equation}
called hereafter the \emph{minimum partition} and the
\emph{maximum partition}, respectively$.$\vspace*{0.05in}

The major goal of the paper is to analyze \emph{quantitative
stability} of the feasible set of the {\em linear} infinite
inequality system (\ref{nominal sys}) under small {\em block
perturbations} of the right-hand side. In more detail, we focus on
characterizing \emph{Lipschitzian behavior} of the feasible
solution map with computing the \emph{exact bound} of Lipschitzian
moduli by using appropriate tools of advanced variational analysis
and generalized differentiation particularly based on
coderivatives. The results obtained for (\ref{nominal sys}) are
then applied to infinite {\em convex} inequalities by means of
their Fenchel-Legendre conjugate linearization.

If no confusion arises, we use the same notation $\Vert \cdot \Vert $ for
the given norm in $X$ and for the corresponding dual norm in $X^{\ast }$
defined by
\begin{equation*}
\left\Vert x^{\ast }\right\Vert :=\sup\limits_{\left\Vert x\right\Vert \leq
1}\left\langle x^{\ast },x\right\rangle \;\text{ for any }\;x^{\ast }\in
X^{\ast },
\end{equation*}
where $\left\langle x^{\ast},x\right\rangle$ stands for the
standard canonical pairing. Our main attention is focused on the
\emph{feasible solution map} $\mathcal{F}_{\mathcal{J}}:l_{\infty
}(J)\rightrightarrows X$ defined by
\begin{equation}
\mathcal{F}_{\mathcal{J}}\left( p\right) :=\big\{x\in X\big|\;x\text{ is a
solution to }\sigma _{\mathcal{J}}(p)\big\}.  \label{feasible j}
\end{equation}

The rest of the paper is organized as follows: Section~2 presents
some basic definitions and key results from variational analysis
and generalized differentiation needed in the sequel. In Section~3
we establish verifiable characterizations of the Lipschitz-like
property of the block-perturbed feasible solution map
(\ref{feasible j}) with precise computing the exact Lipschitzian
bound in terms of the initial data of (\ref{nominal sys}). For
this computation we assume either that $\{a_{t}^{\ast },~t\in T\}$
is bounded in $X^{\ast }$, as in \cite{CLMP09}, or that the Banach
space $X$ of decision variables is reflexive. Section~4 presents
an application of the results obtained for linear systems with
block perturbations to quantitative stability analysis of feasible
solutions to convex inequality systems through their conjugate
linearization. \vspace*{0.05in}

Our notation is basically standard in the areas of variational analysis and
semi-infinite/infinite programming; see, e.g., \cite{GoLo98,mor06a}. Unless
otherwise stated, all the spaces under consideration are Banach. The symbol $%
w^{\ast }$ signifies the weak$^{\ast }$ topology of a dual space, and thus
the weak$^{\ast }$ topological limit corresponds to the weak$^{\ast }$
convergence of nets. Some particular notation will be recalled, if
necessary, in the places where it is introduced.

\section{Preliminaries and First Stability Results}

Given a set-valued mapping $F\colon Z\rightrightarrows Y$ between
Banach spaces $Z$ and $Y$, we say the $F$ is \emph{Lipschitz-like
around} $(\bar{z}, \bar{y})\in \mbox{\rm gph}\,F$, the
\emph{graph} of $F$, with \emph{modulus } $\ell\ge 0$ if there are
neighborhoods $U$ of $\bar{z}$ and $V$ of $\bar{ y}$ such that
\begin{equation}
F(z)\cap V\subset F(u)+\ell \Vert z-u\Vert \mathbb{B}_{Y}\;\text{
for any } \;z,u\in U, \label{eq_Lips}
\end{equation}
where $\mathbb{B}_{Y}$ stands for the closed unit ball in $Y$. The
infimum of moduli $\{\ell \}$ over all the combinations of $\{\ell
,U,V\}$ satisfying (\ref{eq_Lips}) is called the \emph{exact
Lipschitzian bound} of $ F$ around $(\bar{z},\bar{y})$ and is
labeled as $\mbox{\rm lip}\,F(\bar{z}, \bar{y})$.

If $V=Y$ in (\ref{eq_Lips}), this relationship signifies the classical
(Hausdorff) \emph{local Lipschitzian} property of $F$ around $\bar{z}$ with
the \emph{exact Lipschitzian bound} denoted by $\mbox{\rm lip}\, F(\bar{z})$
in this case.

It is worth mentioning that the Lipschitz-like property (also known as the
Aubin or pseudo-Lipschitz property) of an arbitrary mapping $F\colon
Z\rightrightarrows Y$ between Banach spaces is equivalent to other two
fundamental properties in nonlinear analysis while defined for the inverse
mapping $F^{-1}\colon Y\rightrightarrows Z$; namely, to the \emph{metric
regularity} of $F^{-1}$ and to the \emph{linear openness} of $F^{-1}$ around
$(\bar{y},\bar{z})$, with the corresponding relationships between their
exact bounds (see, e.g. \cite{iof,mor06a,rw}). From these relationships we
can easily observe the following representation for the exact Lipschitzian
bound:
\begin{equation}
\mbox{\rm lip}\,F(\bar{z},\bar{y})=\limsup_{(z,y)\rightarrow (\bar{z},\bar{y}%
)}\frac{\mbox{dist}\big(y;F(z)\big)}{\mbox{dist}\big(z;F^{-1}(y)\big)},
\label{eq_quotient}
\end{equation}%
where $\inf \emptyset :=\infty $ (and hence $\mbox{dist}(x;\emptyset
)=\infty $) as usual, and where $0/0:=0$. We have accordingly that $%
\mbox{\rm lip}\,F(\bar{z},\bar{y})=\infty $ if $F$ is not Lipschitz-like
around $(\bar{z},\bar{y})$.

A remarkable fact consists of the possibility to characterize pointwisely
the (derivative-free) Lipschitz-like property of $F$ around $(\bar{z},\bar{y}%
)$---and hence its local Lipschitzian, metric regularity, and linear
openness counterparts---in terms of a dual-space construction of generalized
differentiation called the \emph{coderivative} of $F$ at $(\bar{z},\bar{y}%
)\in \mbox{\rm gph}\,F$. The latter is a positively homogeneous
multifunction $D^{\ast }F(\bar{z},\bar{y})\colon Y^{\ast }\rightrightarrows
Z^{\ast }$ defined by
\begin{equation}
D^{\ast }F(\bar{z},\bar{y})(y^{\ast }):=\big\{z^{\ast }\in Z^{\ast }\big|%
\;(z^{\ast },-y^{\ast })\in N\big((\bar{z},\bar{y});\mbox{\rm gph}\,F\big)%
\big\},\quad y^{\ast }\in Y^{\ast },  \label{cod}
\end{equation}%
where $N(\cdot ;\Omega )$ stands for the collection of generalized
normals to a set at a given point known as the \emph{basic}, or
\emph{limiting}, or \emph{Mordukhovich normal cone}; see, e.g.
\cite{mor76,mor06a,rw,s} and references therein. When both $Z$ and
$Y$ are finite-dimensional, it is proved in \cite{mor93} (cf.\
also \cite[Theorem~9.40]{rw}) that a closed-graph mapping $F\colon
Z\rightrightarrows Y$ is Lipschitz-like around
$(\bar{z},\bar{y})\in\mbox{\rm gph}\,F$ if and only if
\begin{equation}
D^{\ast }F(\bar{z},\bar{y})(0)=\{0\},  \label{cod-cr}
\end{equation}
and the exact Lipschitzian bound of moduli $\{\ell \}$ in (\ref{eq_Lips}) is
computed by
\begin{equation}
\mbox{\rm lip}\,F(\bar{z},\bar{y})=\Vert D^{\ast
}F(\bar{z},\bar{y})\Vert :=\sup \big\{\Vert z^{\ast }\Vert
\;\big|\;z^{\ast }\in D^{\ast }F(\bar{z}, \bar{y})(y^{\ast
}),\;\Vert y^{\ast }\Vert \leq 1\big\}. \label{eq_norm coderiv}
\end{equation}
There is an extension \cite[Theorem~4.10]{mor06a} of the
coderivative criterion \eqref{cod-cr}, via the so-called mixed
coderivative of $F$ at $( \bar{z},\bar{y})$, to the case when both
spaces $Z$ and $Y$ are Asplund (i.e., their separable subspaces
have separable duals) under some additional \textquotedblleft
partial normal compactness" assumption that is automatic in finite
dimensions. Also the aforementioned theorem contains an extension
of the exact bound formula \eqref{eq_norm coderiv} provided that
$Y$ is Asplund while $Z$ is finite-dimensional. Unfortunately,
none of these results is applied in our setting (\ref{feasible j})
when $J$ is infinite; the latter is our standing assumption
needed, in particular, for applications to convex infinite systems
developed in Section~4.

Nevertheless we show in this paper that both \eqref{cod-cr} and
\eqref{eq_norm coderiv} remain valid for
$\mathcal{F}_{\mathcal{J}}\colon l_{\infty }(J)\rightrightarrows
X$ in (\ref{feasible j}) defined by the {\em block-perturbed}
infinite system of linear inequalities (\ref{linear sys j}). The
graph $\mbox{\rm gph}\,\mathcal{F}_{\mathcal{J}}$ of this mapping
is obviously convex, and we can easily verify that it is also
closed with respect to the product topology. If the partition
index set $J$ is infinite, $l_{\infty }(J)$ is an
infinite-dimensional Banach space, which is \emph{never Asplund}.
It is well known from functional analysis (see, e.g., \cite{DS})
that there exists an isometric isomorphism between the topological
dual $l_{\infty }(J)^{\ast }$ and the space $ba(J)$ of additive
and bounded measures on $2^{J}$.

Given a subset $S$ of a normed space, the notation $\mbox{\rm co}\,S$ and $%
\mbox{\rm cone}\,S$ stand for the convex hull and the conic convex hull of $%
S $, respectively. The symbol $\mathbb{R}_{+}$ signifies the interval $\left[
0,\infty \right) $, and by $\mathbb{R}_{+}^{\left( J\right) }$ we denote the
collection of all the functions $\lambda =\left( \lambda _{j}\right) _{j\in
J}\in \mathbb{R}_{+}^{J}$ such that $\lambda _{j}>0$ for only \emph{finitely
many} $j\in J$. As usual, $\mathrm{cl}^{\ast }S$ stands for the weak$^{\ast
} $ ($w^{\ast }$ in brief) topological closure of $S$.

Following the lines in \cite[Theorem~3.2]{CLMP09} and appealing to
the extended Farkas Lemma (see \cite[Lemma~2.1]{CLMP09} and
references therein), we have the following characterization of
$D^{\ast }\mathcal{F}_{\mathcal{J} }\left(0,\overline{x}\right)$,
where we use the notation $\delta _{j}$ for the classical
\emph{Dirac measure} at $j\in J$ given by
\begin{equation*}
\left\langle\delta _{j},p\right\rangle:=p_{j}\text{ for
}p=\left(p_{j}\right)_{j\in J}\in l_{\infty}\left( J\right).
\end{equation*}

\begin{proposition} {\bf (computing coderivatives for linear systems).}
{\normalsize \label{pr coderivJ} Consider any $\overline{x}\in
\mathcal{F}_{\mathcal{J }}\left( 0\right)$ for the mapping
}$\mathcal{F}_{\mathcal{J}}\colon l_{\infty}(J)\rightrightarrows
X$ {\normalsize defined by }\emph{(\ref{feasible j}).}
{\normalsize Then we have }$p^{\ast }\in D^{\ast
}\mathcal{F}_{\mathcal{J}}\left( 0,\overline{x}\right) \left(
x^{\ast }\right) $ if and only if
\begin{equation*}
\left( p^{\ast },-x^{\ast },-\left\langle x^{\ast },\overline{x}
\right\rangle\right)\in\mathrm{cl}^{\ast}\mathrm{cone}\big\{
\left(-\delta_{j},a_{t}^{\ast},b_{t}\right)\big|\; j\in J,~t\in
T_{j}\big\} .
\end{equation*}
\end{proposition}

Let us now define the characteristic set {\normalsize
\begin{equation}
C_{\mathcal{J}}\left( p\right) :=\mbox{\rm co}\,\left\{ \left(
a_{t}^{\ast },b_{t}+p_{j}\right) ,~t\in T_{j},~j\in J\right\}
\subset X^{\ast }\times \mathbb{R} \label{C linear j}
\end{equation}
for $p\in l_{\infty }(J)$. Observe that }$C_{\mathcal{J}}\left(
0\right) $ actually does not depend on $ \mathcal{J}$ but just on
the nominal system (\ref{nominal sys}). For this reason, we denote
in what follows the $C_{\mathcal{J}}\left( 0\right)$ simply by
$C\left(0\right),$ i.e., {\normalsize
\begin{equation*}
C\left( 0\right):=\mbox{\rm co}\,\big\{\left( a_{t}^{\ast
},b_{t}\right),~t\in T\big\}.
\end{equation*}}

{\normalsize We say that the system }$\sigma \left( 0\right) $ in (\ref%
{nominal sys}){\normalsize \ satisfies the \emph{strong Slater
condition} (SSC) if there exists a point $\widehat{x}\in X$ such
that
\begin{equation*}
\sup_{t\in T}\left[ \left\langle a_{t}^{\ast
},\widehat{x}\right\rangle -b_{t}\right]<0.
\end{equation*}
In this case $\widehat{x}$ is called a \emph{strong Slater point}
(SS point in brief) for }$\sigma\left( 0\right)${\normalsize $.$}

\begin{lemma}
{\normalsize \label{Lemma1} {\bf (equivalent descriptions of the
Lipschitz-like property).} Assume that
$\overline{x}\in\mathcal{F}_{\mathcal{J}}\left( 0\right)$. The
following statements are equivalent:}

\emph{(i)}{\normalsize \ $\mathcal{F}$}$_{\mathcal{J}}${\normalsize \ is
Lipschitz-like around $(0,\overline{x});$}

\emph{(ii)}{\normalsize \ }$D^{\ast }\mathcal{F}_{\mathcal{J}}(0,\bar{x}%
)(0)=\{0\};$

\emph{(iii)} ${\normalsize \sigma }\left(0\right) ${\normalsize \
satisfies the SSC}$;${\normalsize \ }

\emph{(iv)}{\normalsize \ $0\in \mathrm{int}(\mbox{\rm dom}\,\mathcal{F_{%
\mathcal{J}})}$}$;${\normalsize \ }

\emph{(v)}{\normalsize \ $\mathcal{F}$}$_{\mathcal{J}}${\normalsize \ is
Lipschitz-like around $(0,x)$ for all $x\in \mathcal{F}_{\mathcal{J}}\left(
0\right) ;$ }

$\emph{(vi)}${\normalsize \ $(0,0)\notin \mbox{\rm cl}\,^{\ast }C\left(
0\right) .$}
\end{lemma}

\begin{proof}
(i)$\Rightarrow $(ii) is a consequence of \cite[Theorem~1.44]{mor06a}
established for general set-valued mappings of closed graph between Banach
spaces. The proof of (ii)$\Rightarrow $(i) follows the lines in the proof of
\cite[Theorem~4.1]{CLMP09}.

In the case of the maximum partition as in (\ref{J min max}) the
equivalence between (iii) and (vi) may be found in, e.g.,
\cite[Theorem~3.1]{GoLoTo96}; see also \cite[Theorem~6.1]{GoLo98}.
Since (iii) and (vi) are not of parametric nature (i.e., their
definitions involve just the nominal system, independently of the
partition under consideration), the equivalence between them holds
true. Moreover, equivalence (iii)$\Longleftrightarrow $(iv) for
the maximum partition trivially entails that
(iii)$\Longrightarrow$(iv) for the arbitrary partition
$\mathcal{J}$, since block perturbations are a particular case of
arbitrary perturbations. The reverse implication
(iv)$\Longrightarrow $(iii) holds by considering a constant
perturbation $p\equiv\varepsilon $ for $\varepsilon
>0$  sufficient small to
guarantee that $p\in {\normalsize \mathrm{int}}${\normalsize
$(\mbox{\rm dom}\,\mathcal{F_{\mathcal{J}})}$} by taking into
account that constant perturbations (corresponding to the minimum
partition) are trivially a particular case of block perturbations.
The equivalences (i)$\Longleftrightarrow $(iv) and (iv)$
\Longleftrightarrow $(v){\normalsize \ }follows from the classical
Robinson-Ursescu theorem. This completes the proof of the
lemma{\normalsize \vspace*{0.05in}}
\end{proof}

{\normalsize The following technical statement is of its own interest while
playing an essential role in proving the main results presented in the
subsequent sections. We keep the convention $0/0:=0$. Observe that this
result is not of parametric nature (i.e., no concept involving perturbation
of }$p${\normalsize \ is used).}

\begin{lemma} {\bf (distance to feasible solutions).}
{\normalsize \label{Lem_distance}}\emph{\cite[Lemma 4.3]{CLMP09}}
{\normalsize \ Assume that the SSC is satisfied for the system
$\sigma _{\mathcal{J}}\left( p\right) $ in }\emph{(\ref{linear sys
j})}{\normalsize \ for }$p\in l_{\infty }\left( J\right)
${\normalsize. Then for any $x\in X$ we have the representation
\begin{equation}
\mathrm{dist}\big(x;\mathcal{F}_{\mathcal{J}}(p)\big)=\sup_{\left(
x^{\ast},\alpha \right) \in {\mathrm{\small cl}}^{\ast
}C_{\mathcal{J}}\left( p\right) }\frac{\left[ \left\langle x^{\ast
},x\right\rangle -\alpha \right] _{+}}{\left\Vert
x^{\ast}\right\Vert }. \label{eq_distance_formula}
\end{equation}
If furthermore the space $X$ is reflexive, then
\begin{equation}
\mathrm{dist}\big(x;\mathcal{F}_{\mathcal{J}}\left( p\right) \big)%
=\sup_{\left( x^{\ast },\alpha \right) \in C_{\mathcal{J}}\left( p\right) }%
\frac{\left[ \left\langle x^{\ast },x\right\rangle -\alpha \right] _{+}}{%
\left\Vert x^{\ast }\right\Vert }.  \label{distC(p)}
\end{equation}}
\end{lemma}

\begin{remark}
\emph{According to the extended Farkas Lemma in \cite[Lemma 2.1]
{CLMP09} the feasibility of }$\sigma _{\mathcal{J}}\left( p\right)
$\emph{ ensures that }$\alpha \le 0$\emph{\ whenever }$\left(
0,\alpha \right)\in{\mathrm{\small
cl}}^{\ast}C_{\mathcal{J}}\left(p\right),$\emph{ and then the
convention }$0/0:=0$ \emph{is applied. Moreover,
\cite[Example~4.4] {CLMP09} shows that the simplified expression
(\ref{distC(p)}) may fail for the nonreflexive Asplund space
}$X=c_{0}$\emph{\ of all sequences converging to zero endowed with
the supremum norm.}
\end{remark}

\section{Quantitative Stability of Linear Systems under Block Perturbations}

The main result of this section is Theorem~\ref{Th_norm lip j},
where an expression for the coderivative norm and the exact
Lipschitzian bound of the feasible solution set mapping of
block-perturbed linear inequality systems is provided under either
the coefficient boundedness $\left\{ a_{t}^{\ast },~t\in T\right\}
$ or the reflexivity of the decision space $X.$ To accomplish
this, we proceed the following chain of technical
lemmas.\vspace*{0.05in}

Recall that $\mathcal{F}_{\mathcal{J}}:l_{\infty
}(J)\rightrightarrows X$ is defined by (\ref{feasible j}) with an
arbitrary Banach decision space $X$ unless otherwise stated.
Moreover, the zero vector or function in all the spaces under
consideration are simply denoted by $0$.

\begin{lemma}{\bf (relationships between exact Lipschitzian bounds
of block-perturbed systems).}\label{Lem_lip refinement} Let
$\overline{x}\in\mathcal{F}_{\mathcal{J}}\left(0\right)$. Then we
have
\begin{equation*}
\mathrm{lip}\,\mathcal{F}_{\min }\left( 0,\overline{x}\right)\le
\mathrm{lip}\,\mathcal{F}_{\mathcal{J}}\left(
0,\overline{x}\right)\le\mathrm{lip} \,\mathcal{F}_{\max }\left(
0,\overline{x}\right)
\end{equation*}
in the notation of \eqref{J min max}.
\end{lemma}

\begin{proof}
Consider the nontrivial case when SSC is satisfied at the nominal
system $\sigma\left(0\right)$; otherwise all the exact
Lipschitzian bounds are $\infty$ according to the equivalence
(i)$\Longleftrightarrow $(iii) in
Lemma~\ref{Lemma1}$)${\normalsize . }Note that the mappings
$\mathcal{F}_{\min },\,\mathcal{F}_{\mathcal{J}}\,$, and
$\mathcal{F}_{\max }$ act in the spaces $\mathbb{R}$, $l_{\infty
}(J)$, and $l_{\infty }(T)$, respectively. For each $\rho\in
\mathbb{R}$ let $p_{\rho }$ be the constant function $p_{\rho
}\equiv\rho$ on $J,$ and for each {\normalsize $p\in l_{\infty
}\left( J\right) $} denote by {\normalsize $p_{T}$ }the constant
by blocks function on $T$ defined as $p_{j}$ on block $T_{j},$
$j\in J.$ Then the proof of the lemma relies on the observation
that {\normalsize
\begin{equation*}
\mathrm{dist}\left( \rho ;\mathcal{F}_{\min }^{-1}\left( x\right)
\right) \geq \mathrm{dist}\left( p_{\rho
};\mathcal{F}_{\mathcal{J}}^{-1}\left( x\right) \right)\;\mbox{
and }\;\mathrm{dist}\left(
p;\mathcal{F}_{\mathcal{J}}^{-1}\left(x\right)\right)\ge\mathrm{dist}\left(
p_{T};\mathcal{F}_{\max }^{-1}\left( x\right)\right)
\end{equation*}
for any $x\in X$. In more details, for the first inequality (and
similarly for the second one) observe that
$\mathcal{F}_{\mathcal{J}}^{-1}\left(x\right)=\varnothing$ yields
$\mathcal{F}_{\min }^{-1}\left( x\right) =\varnothing $. Consider
further the nontrivial case when both sets are nonempty. Thus we
get for some sequence $\left\{\rho
_{r}\right\}_{r\in\mathbb{N}}\subset\mathcal{F}_{\min }^{-1}\left(
x\right)$ that
\begin{equation*}
\mbox{dist}\left(\rho ;\mathcal{F}_{\min }^{-1}\left(x\right)
\right)=\lim_{r\in\mathbb{N}}\left\vert \rho-\rho
_{r}\right\vert=\lim_{r\in \mathbb{N}}\left\Vert p_{\rho }-p_{\rho
_{r}}\right\Vert \geq \mbox{dist} \left( p_{\rho
};\mathcal{F}_{\mathcal{J}}^{-1}\left( x\right)\right)
\end{equation*}
by taking into account that }${\normalsize \rho }${\normalsize
$_{r}\in \mathcal{F}_{\min }^{-1}\left( x\right)$ if and only if
$p_{\rho _{r}}\in \mathcal{F}_{\mathcal{J}}^{-1}\left(x\right)$.}

Finally, we appeal to the Lipschitzian bound representation
(\ref{eq_quotient}) combined with the facts that
$$
\mathcal{F}_{\min }\left(\rho \right)
=\mathcal{F}_{\mathcal{J}}\left( p_{\rho }\right)\;\mbox{ and }\;
\mathcal{F}_{\mathcal{J}}\left(p\right)=\mathcal{F}_{\max}\left(
p_{T}\right),
$$
which thus completes the proof of the lemma.
\end{proof}

\begin{lemma}{\bf (relationship between coderivative norms for block-perturbed systems).}
\label{Lem cod norm j} {\normalsize Take any $\overline{x}\in
\mathcal{F}_{\mathcal{J}}\left( 0\right)$ and consider also the
mapping } $\mathcal{F}_{\min }\colon\mathbb{R} \rightrightarrows
X.$ {\normalsize Then we have the relationship}
\begin{equation}
\left\Vert D^{\ast}\mathcal{F}_{\min }\left(0,\overline{x}\right)
\right\Vert\le\left\Vert D^{\ast}\mathcal{F}_{\mathcal{J}}\left(
0, \overline{x}\right)\right\Vert.\label{eq norm1<j}
\end{equation}
\end{lemma}

\begin{proof}
Observe that {\normalsize $\mathcal{F}_{\mathcal{J}}\left(
0\right)=$}$ \mathcal{F}_{\min }\left( 0\right)$ since both sets
are nothing else but the nominal feasible set. Hence
$\overline{x}\in \mathcal{F}_{\min }\left( 0\right)$. According to
the coderivative norm definition in (\ref{eq_norm coderiv}), pick
arbitrarily $x^{\ast}\in X^{\ast}$ with $\left\Vert x^{\ast
}\right\Vert\le 1$ and consider the nontrivial case when there
exists $\mu \in \mathbb{R}\backslash \{0\}$ with $\mu \in D^{\ast
}\mathcal{F }_{\min }\left( 0,\overline{x}\right) \left( x^{\ast
}\right)$. The coderivative calculation in Proposition~\ref{pr
coderivJ} entails the existence of {\normalsize a net
$\left\{\lambda _{\nu }\right\} _{\nu \in \mathcal{N} } $ with
$\lambda _{\nu }=\left( \lambda _{t\nu }\right) _{t\in T}\in
\mathbb{R}_{+}^{\left(T\right) }$ as }$\nu ${\normalsize $\in
$}$\mathcal{N} ${\normalsize \ satisfying
\begin{equation}
\big(\mu ,-x^{\ast},-\left\langle x^{\ast
},\overline{x}\right\rangle \big) =w^{\ast }\text{-}\lim_{\nu \in
\mathcal{N}}\sum_{t\in T}\lambda _{t\nu }\left( -1,a_{t}^{\ast
},b_{t}\right).\label{701}
\end{equation}
Looking at the first coordinates in (\ref{701}) and setting
}$\gamma _{\nu }:=\sum_{t\in T}\lambda _{t\nu }$, we obtain
\begin{equation}
-\mu =\lim_{\nu \in \mathcal{N}}\gamma _{\nu }>0,  \label{703}
\end{equation}
and hence $\gamma _{\nu }>0$ for $\nu$ sufficiently advanced in
the directed set $\mathcal{N}$; say for all $\nu $ without loss of
generality. This gives us the expression {\normalsize
\begin{equation}
\big(\mu ^{-1}x^{\ast },\left\langle \mu ^{-1}x^{\ast },\overline{x}%
\right\rangle \big)=w^{\ast }\text{-}\lim_{\nu \in \mathcal{N}}\sum_{t\in
T}\gamma _{\nu }^{-1}\lambda _{t\nu }\left( a_{t}^{\ast },b_{t}\right) \in %
\mbox{\rm cl}\,^{\ast }C\left( 0\right).  \label{702}
\end{equation}}

For each $\nu\in\mathcal{N}$ we consider the net $\eta _{\nu
}=\left(\eta _{j\nu }\right)_{j\in J}\in \mathbb{R}_{+}^{\left(
J\right)}$ with $\eta _{j\nu }:=\sum_{t\in T_{j}}\gamma _{\nu
}^{-1}\lambda _{t\nu }$, which obviously satisfies the condition
$\sum_{j\in J}\eta _{j\nu }=1.$ {\normalsize Since the net $
\{\sum_{j\in J}\eta _{j\nu }\left(-\delta _{j}\right) \}_{\nu \in
\mathcal{N }}$ is contained in }$\mathbb{B}${\normalsize
$_{l_{\infty }\left( J\right) ^{\ast }}$, the classical
Alaoglu-Bourbaki theorem ensures that a certain subnet (indexed
without relabeling by $\nu \in \mathcal{N}$) weak$^{\ast }$
converges to some $p^{\ast }\in l_{\infty }\left( J\right)^{\ast
}$ with $ \left\Vert p^{\ast }\right\Vert\le 1$. Denoting by $e\in
l_{\infty }\left( J\right)$ the function whose coordinates are
identically one, we get
\begin{equation*}
\left\langle p^{\ast},-e\right\rangle =\lim_{\nu \in
\mathcal{N}}\sum_{t\in j}\eta _{j\nu }=1,
\end{equation*}
and hence $\left\Vert p^{\ast}\right\Vert =1$. Appealing now to
(\ref{702}) gives us, for the subnet under consideration
(}recalling the definition of $\eta _{j\nu })${\normalsize , the
equality
\begin{equation*}
\left( p^{\ast },\mu ^{-1}x^{\ast },\left\langle \mu ^{-1}x^{\ast
}, \overline{x}\right\rangle \right) =w^{\ast }\text{-}\lim_{\nu
\in \mathcal{N} }\sum_{j\in J}\sum_{t\in T_{j}}\gamma _{\nu
}^{-1}\lambda _{t\nu }\left( -\delta _{j},a_{t}^{\ast
},b_{t}\right).
\end{equation*}
Employing further the coderivative description from
Proposition~\ref{pr coderivJ} yields
\begin{equation*}
p^{\ast }\in D^{\ast }\mathcal{F}_{\mathcal{J}}\left(
0,\overline{x}\right) \left( -\mu ^{-1}x^{\ast }\right).
\end{equation*}
Recalling (\ref{703}), the positive homogeneity of the
coderivative ensures }
\begin{equation*}
-\mu p^{\ast }\in D^{\ast }\mathcal{F}_{\mathcal{J}}\left( 0,\overline{x}%
\right)\left( x^{\ast }\right),
\end{equation*}
{\normalsize which implies by definition of the coderivative norm
in (\ref {eq_norm coderiv}) that
\begin{equation*}
\left\Vert D^{\ast}\mathcal{F}_{\mathcal{J}}\left(
0,\overline{x}\right) \right\Vert \ge \left\Vert -\mu p^{\ast
}\right\Vert =-\mu =\left\vert \mu \right\vert.
\end{equation*}
Since }$\mu \in D^{\ast }\mathcal{F}_{\min }\left(
0,\overline{x}\right) \left( x^{\ast }\right) ${\normalsize \ was
chosen arbitrarily, we arrive at (\ref{eq norm1<j}) and thus
complete the proof of the lemma.}
\end{proof}

\begin{remark}
\label{ReM_SS_point}\emph{In the sequel we adopt the convention}
$\sup \varnothing :=0,$ \emph{which makes sense while dealing with
nonnegative numbers. Observe that under this convention we have
for a SS point }$\overline{x}$ \emph{of }$\sigma\left( 0\right)$
\emph{ the equality}
\begin{equation*}
\sup \left\{ \left\Vert u^{\ast }\right\Vert
^{-1}\Big|\;\big(u^{\ast },\left\langle u^{\ast
},\overline{x}\right\rangle \big)\in \mbox{\rm cl} \,^{\ast
}C\left( 0\right) \right\}=0.
\end{equation*}
\emph{In fact, it is easy to check that for a SS point
}$\overline{x}$ \emph{of }$\sigma \left( 0\right)$ \emph{there is
no element }$u^{\ast}\in X^{\ast }$\emph{\ satisfying}
$\big(u^{\ast },\left\langle u^{\ast },\overline{x}\right\rangle
\big)\in \mbox{\rm cl} \,^{\ast }C\left( 0\right) .$ \emph{Note
that the reciprocal is not true in general. To illustrate it,
consider the system } $\sigma \left( 0\right):=\{tx\le 1/t;$
$t=1,2,\ldots\}$ {\emph{in } $\mathbb{R}$. \emph{On one hand,
observe that} $\overline{x}=0$ \emph{is not a SS point. On the
other hand, we have} $\{u^{\ast }\in \mathbb{R}\big|\;\big(u^{\ast
},\left\langle u^{\ast },\overline{x} \right\rangle \big)\in
\mbox{\rm cl}\,^{\ast }C\left( 0\right)\}=\varnothing$.}
\end{remark}

\begin{remark}
\label{ReM_SSC failure}\emph{If SSC fails at }$\sigma \left(
0\right)$, \emph{then Lemma~\ref{Lemma1} ensures that }$\left(
0,0\right) \in \mbox{\rm cl} \,^{\ast }C\left( 0\right)$.
\emph{Under the convention }$0^{-1}:=\infty$ \emph{we have in this
case that}
\begin{equation*}
\sup \left\{ \left\Vert u^{\ast }\right\Vert
^{-1}\Big|\;\big(u^{\ast },\left\langle u^{\ast
},\overline{x}\right\rangle \big)\in \mbox{\rm cl} \,^{\ast
}C\left( 0\right) \right\}=\infty .
\end{equation*}
\end{remark}

\begin{lemma}\label{Lem norm cod 1 formula} {\bf (lower estimate of
the coderivative norm for the minimum partition). }{\normalsize
Consider the mapping} $\mathcal{F}_{\min }\colon
\mathbb{R}\rightrightarrows X$ {\normalsize and pick
$\overline{x}\in \mathcal{F}_{\min }\left( 0\right)$}.
{\normalsize Then we have the estimate
\begin{equation}
\sup \left\{ \left\Vert u^{\ast }\right\Vert ^{-1}\Big|\;\big(u^{\ast
},\left\langle u^{\ast },\overline{x}\right\rangle \big)\in \mbox{\rm cl}%
\,^{\ast }C\left( 0\right) \right\} \leq \left\Vert D^{\ast }\mathcal{F}%
_{\min }\left( 0,\overline{x}\right) \right\Vert .  \label{eq norm F min}
\end{equation}}
\end{lemma}

\begin{proof}
Let us see first that $\left\Vert D^{\ast }\mathcal{F}_{\min
}\left( 0, \overline{x}\right)\right\Vert =\infty $ provided that
the SSC fails at $\sigma \left( 0\right) ${\normalsize. Indeed, in
this case Lemma~\ref{Lemma1} yields that }$ \left(0,0\right)\in
\mbox{\rm cl}\,^{\ast}C\left( 0\right)$, which implies the
existence of {\normalsize a net $\left\{ \lambda _{\nu }\right\}
_{\nu \in \mathcal{N}}$ with $\lambda _{\nu }=\left( \lambda
_{t\nu }\right) _{t\in T}\in \mathbb{R}_{+}^{\left( T\right)}$ and
}$\sum_{t\in T}\lambda _{t\nu }=1$ {\normalsize as }$\nu
${\normalsize $\in $}$\mathcal{N}$ {\normalsize \ satisfying
\begin{equation*}
\big(0,0\big)=w^{\ast }\text{-}\lim_{\nu \in
\mathcal{N}}\sum_{t\in T}\lambda _{t\nu }\left( a_{t}^{\ast
},b_{t}\right).
\end{equation*}
The latter obviously entails that $\big(-1,0,0\big)=w^{\ast
}$-$\lim_{\nu \in \mathcal{ N}}\sum_{t\in T}\lambda _{t\nu }\left(
-1,a_{t}^{\ast },b_{t}\right),$ i.e., by Proposition~\ref{pr
coderivJ} we get}
\begin{equation*}
-1\in D^{\ast }\mathcal{F}_{\min }\left( 0,\overline{x}\right)
\left( 0\right).
\end{equation*}
Since $D^{\ast}\mathcal{F}_{\min }\left( 0,\overline{x}\right)$ is
positively homogeneous, the coderivative norm definition gives us
the claimed condition $\left\Vert D^{\ast }\mathcal{F}_{\min
}\left( 0,\overline{x}\right) \right\Vert=\infty$.

{\normalsize Now we consider the nontrivial case when the SSC
holds at $\sigma \left( 0\right) $ and the set of elements
$u^{\ast}\in X^{\ast }$ with $\left( u^{\ast },\left\langle
u^{\ast },\overline{x}\right\rangle \right) \in \mbox{\rm
cl}\,^{\ast }C\left( 0\right)$ is nonempty. Take such an element}
$u^{\ast }${\normalsize. and observe that the fulfillment of the
SSC for $\sigma \left( 0\right) $ ensures that $u^{\ast}\ne 0$
according to Lemma~\ref{Lemma1}. By the choice of }$u^{\ast
}${\normalsize , find a net $\left\{ \lambda _{\nu }\right\} _{\nu
\in \mathcal{N}}$ with $\lambda _{\nu }=\left( \lambda _{t\nu
}\right) _{t\in T}\in \mathbb{R}_{+}^{\left( T\right) }$ and $
\sum_{t\in T}\lambda _{t\nu }=1$ as }$\nu ${\normalsize $\in
$}$\mathcal{N}$ {\normalsize \ satisfying
\begin{equation}
\big(u^{\ast },\left\langle u^{\ast },\overline{x}\right\rangle
\big) =w^{\ast }\text{-}\lim_{\nu \in \mathcal{N}}\sum_{t\in
T}\lambda _{t\nu }\left( a_{t}^{\ast },b_{t}\right) .
\label{eq_502}
\end{equation}
Then (\ref{eq_502}) can be trivially rewritten as }
\begin{equation*}
\big(-1,u^{\ast },\left\langle u^{\ast },\overline{x}\right\rangle \big)%
=w^{\ast }\text{-}\lim_{\nu {\normalsize \in
}\mathcal{N}}\sum_{t\in T}\lambda _{t\nu }\left( -1,a_{t}^{\ast
},b_{t}\right),
\end{equation*}
{\normalsize which implies that }${\normalsize -}1\in D^{\ast
}\mathcal{F}_{\min }\left( 0,\overline{x}\right) \left( -u^{\ast
}\right) .${\normalsize \ Hence hence
\begin{equation*}
-\left\Vert u^{\ast }\right\Vert ^{-1}\in D^{\ast }\mathcal{F}_{\min }\left(
0,\overline{x}\right) \left( -\left\Vert u^{\ast }\right\Vert ^{-1}u^{\ast
}\right) ,
\end{equation*}
which ensures by the definition of the coderivative norm that
\begin{equation*}
\left\Vert D^{\ast }\mathcal{F}_{\min }\left( 0,\overline{x}\right)
\right\Vert \geq \left\Vert u^{\ast }\right\Vert ^{-1}.
\end{equation*}
Since $u^{\ast}$ was chosen arbitrarily from those satisfying
$\left( u^{\ast },\left\langle u^{\ast },\overline{x}\right\rangle
\right) \in \mbox{\rm cl}\,^{\ast }C\left( 0\right) $, we arrive
at the lower estimate ( \ref{eq norm F min}) for the coderivative
norm and thus complete the proof of this lemma.}
\end{proof}\vspace*{0.05in}

Now we are ready to establish the main result of this section.

\begin{theorem} {\bf (evaluation of coderivative norms for
block-perturbed systems).} {\normalsize \label{Th_norm lip j} For
any $\overline{x}\in \mathcal{F}_{ \mathcal{J}}\left( 0\right)$}
{\normalsize we have the relationships
\begin{align*}
\sup \left\{ \left\Vert u^{\ast }\right\Vert
^{-1}\Big|\;\big(u^{\ast },\left\langle u^{\ast
},\overline{x}\right\rangle \big)\in \mbox{\rm cl} \,^{\ast
}C\left( 0\right) \right\} & \leq \left\Vert D^{\ast }\mathcal{F}
_{\min }\left( 0,\overline{x}\right) \right\Vert \leq \left\Vert
D^{\ast }\mathcal{F}_{\mathcal{J}}\left( 0,\overline{x}\right) \right\Vert \\
& \le\mathrm{lip}\,\mathcal{F}_{\mathcal{J}}\left(
0,\overline{x}\right)\le\mathrm{lip}\,\mathcal{F}_{\max }\left(
0,\overline{x}\right) .
\end{align*}}
{\normalsize Furthermore, if either the coefficient set
}$\{a_{t}^{\ast },~t\in T\}${\normalsize \ is bounded in $X^{\ast
}$ or the space }$X$ {\normalsize is reflexive, then all the above
inequalities hold as equalities.}
\end{theorem}

\begin{proof}
{\normalsize The lower bound estimate
\begin{equation}
\Vert D^{\ast }\mathcal{F}_{\mathcal{J}}(0,\bar{x})\Vert \leq \mbox{\rm lip}%
\,\mathcal{F}_{\mathcal{J}}(0,\bar{x})  \label{144}
\end{equation}
is proved in \cite[Theorem~1.44]{mor06a} for general set-valued
mappings between Banach spaces. Now apply (in this order)
Lemmas~\ref{Lem norm cod 1 formula}, \ref{Lem cod norm j}, formula
(\ref{144}), and Lemma~\ref{Lem_lip refinement} to obtain the
claimed chain of inequalities.}

{\normalsize Consider first the case when} the set $\{a_{t}^{\ast
},~t\in T\}${\normalsize \ is bounded in $X^{\ast }.$}
{\normalsize Then applying \cite[Theorem~4.6]{CLMP09} adapted to
the current notation gives us}
\begin{equation}
\mathrm{lip}\,\mathcal{F}_{\max }\left( 0,\overline{x}\right) \leq \sup
\left\{ \left\Vert u^{\ast }\right\Vert ^{-1}\Big|\;\big(u^{\ast
},\left\langle u^{\ast },\overline{x}\right\rangle \big)\in \mbox{\rm cl}%
\,^{\ast }C\left( 0\right) \right\}  \label{eq_upper_bound}
\end{equation}
{\normalsize in the nontrivial case when SSC holds at }$\sigma
\left( 0\right)$; {\normalsize Remark~\ref{ReM_SSC failure}.}

{\normalsize To finish the proof of this theorem, it remains to
establish the same inequality (\ref{eq_upper_bound}), again in the
nontrivial case when the SSC holds at }$\sigma \left( 0\right)
${\normalsize ,\ under the assumption that }$X$ {\normalsize is
{\em reflexive}, in which case the classical Mazur theorem allows
us to replace the weak$^{\ast }$ closure }$ \mbox{\rm cl}\,^{\ast
}C\left( 0\right) $ {\normalsize of the convex set $C\left(
0\right) $ by its norm closure }$\mbox{\rm cl}\,C\left( 0\right) $
{\normalsize. Arguing by contradiction to \eqref{eq_upper_bound},
find }${\normalsize \beta }$ {\normalsize $>0$ such that
\begin{equation}
\mbox{\rm lip}\,\mathcal{F}_{\max }\left( 0,\overline{x}\right)
>\beta >\sup \left\{ \left\Vert u^{\ast }\right\Vert
^{-1}\Big|\;\left( u^{\ast },\left\langle u^{\ast
},\overline{x}\right\rangle \right) \in \mbox{\rm cl} \,C\left(
0\right) \right\}.  \label{eq_103}
\end{equation}\
According to (\ref{eq_quotient}) and the first inequality in
(\ref{eq_103}), there are sequences $p_{r}=\left( p_{tr}\right)
_{t\in T}\rightarrow 0$ and $ x_{r}\rightarrow \overline{x}$ along
which
\begin{equation}
\mbox{dist}\big(x_{r};\mathcal{F}_{\max }(p_{r})\big)>\beta
\,\mbox{dist} \big(p_{r};\mathcal{F}_{\max }^{-1}\left(
x_{r}\right) \big)\;\text{ for all }\;r\in \mathbb{N}.
\label{eq_104}
\end{equation}
By the SSC at $\sigma \left( 0\right)$ we have due to
Lemma~\ref{Lemma1} that $\mathcal{F}_{\max }\left( p_{r}\right)
\ne\varnothing $ for $r\in\N$ sufficiently large; say for all
$r\in\N$ without loss of generality. The imposed SSC at $\sigma
\left( 0\right)$ is also equivalent to the inner/lower
semicontinuity of $\mathcal{F}$}$_{\max }${\normalsize \ around
}$\overline{p }=0${\normalsize \ by \cite[Theorem~5.1]{DGL08},
which entails that
\begin{equation}
\lim_{r\rightarrow \infty }\mbox{dist}\big(x_{r};\mathcal{F}_{\max
}\left( p_{r}\big)\right)=0.\label{semicontinf}
\end{equation}
Moreover, it follows from (\ref{eq_104}) that the quantity
\begin{eqnarray}
\mbox{dist}\big(p_{r};\mathcal{F}_{\max }^{-1}\left( x_{r}\right)
\big)&=&\sup_{t\in T}\left[ \left\langle a_{t}^{\ast
},x_{r}\right\rangle
-b_{t}-p_{tr}\right]_{+}\label{nueva} \\
&=&\sup_{\left( x^{\ast },\alpha \right) \in C_{\max }\left(
p_{r}\right)} \left[ \left\langle x^{\ast },x_{r}\right\rangle
-\alpha \right]_{+}\notag
\end{eqnarray}
is finite. We may assume without loss of generality that the SSC
holds at }$ \sigma _{\max }\left( p_{r}\right) $ {\normalsize for
all} $r.$ {\normalsize Then it follows from
Lemma~\ref{Lem_distance} that
\begin{equation*}
\mbox{dist}\big(x_{r};\mathcal{F}_{\max }\left( p_{r}\right)\big)
=\sup_{\left( x^{\ast },\alpha \right) \in C_{\max }\left(
p_{r}\right) } \frac{\left[ \left\langle x^{\ast
},x_{r}\right\rangle -\alpha \right]_{+}}{ \left\Vert x^{\ast
}\right\Vert },\quad r=1,2,\ldots.
\end{equation*}
This allows us to find $\left( x_{r}^{\ast },\alpha _{r}\right)
\in C_{\max }\left( p_{r}\right) $ as $r\in\N$ satisfying
\begin{equation}
0<\mbox{dist}\big(x_{r},\mathcal{F}_{\max }\left( p_{r}\right)
\big)-\frac{ \left\langle x_{r}^{\ast },x_{r}\right\rangle -\alpha
_{r}}{\left\Vert x_{r}^{\ast }\right\Vert }<\frac{1}{r}.
\label{unosobrer}
\end{equation}
Furthermore, by (\ref{eq_104}) and (\ref{nueva}) we can choose
$\left( x_{r}^{\ast },\alpha _{r}\right)$ in such a way that
\begin{equation}
\beta \,\mbox{dist}\big(p_{r};\mathcal{F}_{\max }^{-1}\left(
x_{r}\right) \big)<\frac{\left\langle x_{r}^{\ast
},x_{r}\right\rangle -\alpha _{r}}{ \left\Vert x_{r}^{\ast
}\right\Vert }\leq \frac{\mbox{dist}\big(p_{r}; \mathcal{F}_{\max
}^{-1}\left( x_{r}\right) \big)}{\left\Vert x_{r}^{\ast
}\right\Vert}.  \label{clave}
\end{equation}
Since dist$(p_{r};\mathcal{F}_{\max }^{-1}\left( x_{r}\right) )>0$
(otherwise both members of (\ref{eq_104}) would be zero), we
deduce from ( \ref{clave}) that
\begin{equation*}
\Vert x_{r}^{\ast }\Vert <\frac{1}{\beta }\;\mbox{ for all
}\;r=1,2,\ldots,
\end{equation*}
and thus, by the weak$^{\ast }$ sequential compactness of the unit
ball in duals to reflexive spaces, select a subsequence $\left\{
x_{r_{k}}^{\ast }\right\} _{k\in \mathbb{N}}$, which weak$^{\ast
}$ converges to some $ x^{\ast }\in X^{\ast }$ satisfying
$\left\Vert x^{\ast }\right\Vert \leq 1/$} ${\normalsize \beta
}${\normalsize . Then we get from (\ref{semicontinf}) and
(\ref{unosobrer}) that
\begin{equation*}
\lim_{k\in \mathbb{N}}\frac{\left\langle x_{r_{k}}^{\ast
},x_{r_{k}}\right\rangle -\alpha _{r_{k}}}{\left\Vert x_{r_{k}}^{\ast
}\right\Vert }=0,
\end{equation*}
which implies in turn that
\begin{equation*}
\lim_{k\in \mathbb{N}}\big(\left\langle x_{r_{k}}^{\ast
},x_{r_{k}}\right\rangle -\alpha _{r_{k}}\big)=0.
\end{equation*}
Since the sequence $\left\{ x_{r_{k}}\right\} _{k\in \mathbb{N}}$
converges in norm to $\overline{x}$, the latter implies that
\begin{equation*}
\lim_{k\in \mathbb{N}}\alpha _{r_{k}}=\lim_{k\in \mathbb{N}}\left\langle
x_{r_{k}}^{\ast },x_{r_{k}}\right\rangle=\left\langle x^{\ast },\overline{x}%
\right\rangle.
\end{equation*}
Taking into account that for each $k\in\mathbb{N}$ we have $\left(
x_{r_{k}}^{\ast },\alpha _{r_{k}}\right) \in C_{\max }\left(
p_{r_{k}}\right) $, there exist $\lambda _{r_{k}}=(\lambda
_{tr_{k}})_{t\in T}$ such that $\lambda _{tr_{k}}\ge 0$, only
finitely many of them are positive,
\begin{equation*}
\sum_{t\in T}\lambda _{tr_{k}}=1,\;\mbox{ and
}\;(x_{r_{k}}^{\ast },\alpha _{r_{k}})=\sum_{t\in T}\lambda _{tr_{k}}\left(
a_{t}^{\ast },b_{t}+p_{tr_{k}}\right) ,\quad k\in \mathbb{N}.
\end{equation*}
Combining all the above gives us the relationships
\begin{eqnarray*}
\big(x^{\ast },\left\langle x^{\ast },\overline{x}\right\rangle
\big)&=&w^{\ast }\text{-}\lim_{k\in \mathbb{N}}(x_{r_{k}}^{\ast
},\alpha _{r_{k}})
\\
&=&w^{\ast }\text{-}\lim_{k\in \mathbb{N}}\sum_{t\in T}\lambda
_{tr_{k}}\left( a_{t}^{\ast },b_{t}+p_{tr_{k}}\right) \\
&=&w^{\ast }\text{-}\lim_{k\in \mathbb{N}}\sum_{t\in T}\lambda
_{tr_{k}}\left( a_{t}^{\ast },b_{t}\right) \in \mbox{\rm cl}\,C\left(
0\right) ,
\end{eqnarray*}
where the last equality comes from }$\lim_{k\rightarrow \infty
}\left\Vert p_{r_{k}}\right\Vert =0.$ {\normalsize Observe finally
that $x^{\ast}\ne 0$ because, by Lemma~\ref{Lemma1}, the linear
infinite system $\sigma \left( 0\right) $ satisfies the SSC. This
allows us to conclude that
\begin{equation*}
\sup \left\{ \left\Vert u^{\ast }\right\Vert ^{-1}\big|\;\big(u^{\ast
},\left\langle u^{\ast },\overline{x}\right\rangle \big)\in \mbox{\rm cl}%
\,C\left( 0\right) \right\} \geq \left\Vert x^{\ast }\right\Vert ^{-1}\geq
\beta ,
\end{equation*}
which contradicts (\ref{eq_103}) and thus completes the proof of the
theorem. \medskip }
\end{proof}

{\normalsize We finish this section with a discussion about some
consequences of the boundedness assumption on the coefficient set
$\left\{a_{t}^{\ast }\;|\;t\in T\right\} \subset X^{\ast }$. First
observe that this assumption yields that only $\varepsilon
$-active indices are relevant in the computation of the supremum
of the previous theorem. The following proposition provides a
useful representation of the characteristic set} $\left\{
\big(u^{\ast },\left\langle u^{\ast },\overline{x}\right\rangle
\big)\in \mbox{\rm cl} \,^{\ast }C\left( 0\right) \right\} ,$
{\normalsize which may be rewritten as }${\normalsize \left\{
\left( \overline{x},-1\right) \right\} ^{\bot }\cap
\mathrm{cl}^{\ast }C\left( 0\right) ,}${\normalsize \ in terms of
the sets }
\begin{equation*}
{\normalsize T_{\varepsilon}\left( \overline{x}\right):=\big\{t\in T\big|%
\;\left\langle a_{t}^{\ast },\overline{x}\right\rangle\ge
b_{t}-\varepsilon \big\},\ \ \varepsilon 0.}
\end{equation*}

\begin{proposition} {\bf (limiting representation of the
characteristic set).} {\normalsize \label{pr} Assume that the
coefficient set $\left\{ a_{t}^{\ast }\;|\;t\in T\right\}$ is
bounded in $X^{\ast }$. Then given $\overline{x} \in
\mathcal{F}_{\mathcal{J}}\left( 0\right) $, we have the
representation
\begin{equation}
\big\{\left( \overline{x},-1\right) \big\}^{\bot }\cap \mbox{\rm
cl}\,^{\ast }C\left( 0\right) =\bigcap_{\varepsilon >0}\mbox{\rm
cl}\,^{\ast } \mbox{\rm co}\,\big\{\left( a_{t}^{\ast
},b_{t}\right) \big|\;t\in T_{\varepsilon }\left(
\overline{x}\right) \big\}. \label{eq_intersect}
\end{equation}}
\end{proposition}

\begin{proof}
{\normalsize It follows the lines of justifying Step~1 in the
proof of \cite[ Theorem~1]{CGP08}. Note that both sets in
(\ref{eq_intersect}) are nonempty if and only if $\overline{x}$ is
not a strong Slater point for $\sigma \left( 0\right)$; see
Remark~\ref{ReM_SS_point}.}
\end{proof}\vspace*{0.05in}

{\normalsize Observe that in the continuous case considered in
\cite{CDLP05} (where $T$ is assumed to be a compact Hausdorff
space, $X=\mathbb{R}^{n}$, and the mapping
$t\mapsto\left(a_{t}^{\ast},b_{t}\right)$ is continuous on $T$)
representation (\ref{eq_intersect}) reads as
\begin{equation*}
\big\{\left(\overline{x},-1\right)\big\}^{\bot}\cap C\left(
0\right)=\mbox{\rm co}\,\big\{\left(a_{t}^{\ast },b_{t}\right)
\big|\;\ t\in T_{0}\left(\overline{x}\right) \big\}.
\end{equation*}}

{\normalsize The following example shows that the statement of
Proposition~\ref{pr} is \emph{no longer valid} without the
boundedness assumption on $\left\{a_{t}^{\ast}|\;t\in T\right\}$
and that in the exact bound expression of Theorem~\ref{Th_norm lip
j} via }$\sup \left\{\left\Vert u^{\ast }\right\Vert
^{-1}\big|\;\big(u^{\ast },\left\langle u^{\ast
},\overline{x}\right\rangle \big)\in \mbox{\rm cl}\,^{\ast
}C\left( 0\right)\right\} ${\normalsize \ the set $ \mbox{\rm
cl}\,^{\ast }C\left( 0\right)$ cannot be replaced by $\mathrm{cl}
^{\ast }\mbox{\rm co}\,\left\{ \left( a_{t}^{\ast },b_{t}\right)
\mid t\in T_{\varepsilon }\left( \overline{x}\right)\right\}$ for
some \emph{small }$ \varepsilon>0$; i.e., it is not sufficient to
consider just $\varepsilon $-active constraints.}

\begin{example} {\bf (coefficient boundedness is essential).}
{\normalsize \emph{Consider the countable linear system in
}$\mathbb{R}^{2}$:
\begin{equation*}
\sigma \left( p\right)=\left\{
\begin{array}{ll}
\left( -1\right) ^{t}tx_{1}\leq 1+p_{t}, & t=1,2,\dots, \\
x_{1}+x_{2}\leq 0+p_{0}, & t=0%
\end{array}
\right\}.
\end{equation*}
\emph{The reader can easily check that} \emph{for
}$\overline{x}=0\in\mathbb{R}^{2}$ \emph{and }$0\le\varepsilon<1$
\emph{ we have }
\begin{equation*}
\mathrm{co}\big\{\left( a_{t}^{\ast },b_{t}\right) \big|\;\ t\in
T_{\varepsilon }\left( \overline{x}\right) \big\}=\big\{\left(
1,1,0\right) \big\}\;\rm{ and}
\end{equation*}
\begin{equation*}
\big\{\left( \overline{x},-1\right) \big\}^{\bot }\cap \mbox{\rm cl}\,^{\ast
}C\left( 0\right) =\big\{\left( \alpha ,1,0\right) ,\;\alpha \in \mathbb{R}%
\big\}
\end{equation*}
\emph{It follows furthermore that}
\begin{equation*}
\mathcal{F}_{\max }\left( p\right) =\left\{ 0\right\} \times
\left( -\infty ,p_{0}\right] \;\text{\emph{\ whenever
}}\;\left\Vert p\right\Vert \le 1,
\end{equation*}
\emph{which easily implies that } $\mbox{\rm
lip}\,\mathcal{F}_{\max }\left( 0,\overline{x}\right) =1$\emph{.
Observe however that }$\mbox{\rm lip}\, \mathcal{F}_{\max }\left(
0,\overline{x}\right)$ }\emph{cannot be computed through
}$T_{\varepsilon }\left( \overline{x}\right) $ {\normalsize \emph{
for }$0<\varepsilon $$<1;$ \emph{in fact}
\begin{equation*}
\max \left\{ \left\Vert u^{\ast }\right\Vert ^{-1}\big|\;\left(
u^{\ast },\left\langle u^{\ast },\overline{x}\right\rangle \right)
\in \mathrm{cl} ^{\ast }\mathrm{co}\big\{\left( a_{t}^{\ast
},b_{t}\right) \big|\;t\in T_{\varepsilon }\left(
\overline{x}\right) \big\}\right\} =\frac{1}{\sqrt{2}}.
\end{equation*}}
\end{example}
{\normalsize As mentioned above, it is clear that $\left\{ \left(
\overline{x},-1\right) \right\}^{\bot }\cap \mathrm{cl}^{\ast
}C\left( 0\right) =\varnothing $ when $\overline{x}$ is a SS point
for $\sigma \left( 0\right) $. According to
\cite[Lemma~3.4]{CLMP09}, if $\left\{ a_{t}^{\ast }\;|\;t\in
T\right\} $ is bounded and }$\overline{x}$ {\normalsize is not a
SS point for $\sigma \left( 0\right) $, the set $\left\{ \left(
\overline{x} ,-1\right) \right\} ^{\bot }\cap \mathrm{cl}^{\ast
}C\left( 0\right) $ is nonempty and} $w^{\ast
}${\normalsize-compact in $X^*$}. If in addition the SSC holds at
$\sigma \left( 0\right),$ then the latter set does not contain the
origin and the supremum in Theorem~\ref{Th_norm lip j} becomes a
maximum.

\section{Applications to Convex Systems}

In this section we apply the results above to analyze the
quantitative stability of infinite convex inequality systems by
using the linearization procedure via the Fenchel-Legendre
conjugate. This procedure splits each convex inequality into a
block of linear ones so that a natural perturbation framework for
the linearized system is a block perturbation setting. In what
follows we consider the \emph{parameterized convex inequality
system} given by
\begin{equation}
\sigma (p):=\big\{f_{j}\left( x\right) \leq p_{j},\ j\in J\big\},
\label{eq_conv_sys}
\end{equation}
where $J$ is an arbitrary \emph{index set}, $x\in X$ is a
\emph{decision variable} selected from a general Banach space $X$
with its topological dual $X^{\ast}$, and where the functions
$f_{j}:X\rightarrow \overline{\mathbb{R}}:=\mathbb{R}\cup \{\infty
\},$ $j\in J$, are proper lower semicontinuous (lsc) and convex.
As above, the functional parameter $p$ belongs to the Banach space
$l_{\infty }(J)$ and the zero function $ \overline{p}=0$ is
regarded as the nominal parameter.

Hereafter we denote by $\mathcal{F}$ the \emph{feasible solution
map} of ( \ref{eq_conv_sys}); i.e., $\mathcal{F}:l_{\infty
}(J)\rightrightarrows X$ is defined by
\begin{equation}
\mathcal{F}(p):=\big\{x\in X\big|\;x\text{ is a solution to
}\sigma (p)\big\} .\label{eq_fsm}
\end{equation}

The convex system $\sigma (p)$ with $p\in l_{\infty }(J)$ can be
\emph{ linearized} by using the \emph{Fenchel-Legendre conjugate}
$f_{j}^{\ast }:X^{\ast }\rightarrow \overline{\mathbb{R}}$ for
each function $f_{j}$ given by
\begin{equation*}
f_{j}^{\ast }\left( u^{\ast }\right) :=\sup \big\{\left\langle
u^{\ast },x\right\rangle -f_{j}\left( x\right) \big|\;x\in
X\big\}=\sup \big\{\left\langle u^{\ast },x\right\rangle
-f_{j}\left( x\right) \big|\;x\in \mathrm{dom}f_{j}\big\},
\end{equation*}
where $\mathrm{dom}f_{j}:=\left\{ x\in X\mid f_{j}\left(
x\right)<\infty \right\} $ is the effective domain of $f_{j}$.
Specifically, under the current assumptions on each $f_{j}$ its
conjugate $f_{j}^{\ast }$ is also a proper lsc convex function
such that
\begin{equation*}
f_{j}^{\ast \ast }=f_{j}\;\mbox{ on }\;X\;\mbox{ with
}\;f_{j}^{\ast \ast }:=\left( f_{j}^{\ast }\right) ^{\ast}.
\end{equation*}
In this way, for each $j\in J$, the inequality $f_{j}\left(
x\right)\le p_{j}$ turns out to be equivalent to the linear system
\begin{equation*}
\left\{ \left\langle u^{\ast },x\right\rangle -f_{j}^{\ast }\left( u^{\ast
}\right) \leq p_{j},\text{ }u^{\ast }\in \mathrm{dom}f_{j}^{\ast }\right\}
\end{equation*}
in the sense that they have the same solution sets.

In order to link to the notation of the previous sections, put
\begin{equation*}
T:=\left\{\left(j,u^{\ast }\right)\in J\times X^{\ast
}\;|\;u^{\ast}\in\mathrm{dom}f_{j}^{\ast}\right\}
\end{equation*}
and note that $T$ is partitioned as
\begin{equation}
T=\bigcup_{j\in J}T_{j},\qquad\text{where }T_{j}:=\{j\}\times
\mathrm{dom} f_{j}^{\ast }.\label{eq Tj convex}
\end{equation}
In this way the right-hand side perturbations on the nominal
convex system $\sigma (0)$ correspond to block perturbations of
the linearized nominal system $\sigma_{\mathcal{J}}(0)$ with the
partition $\mathcal{J}:=\left\{ T_{j}\mid j\in J\right\}$. It is
important to realize  to this end that $\mathcal{F}$ and
$\mathcal{F}_{\mathcal{J}}$ are \emph{exactly the same mapping}$.$

Recall that the \emph{epigraph} of a function $h\colon
X\rightarrow\overline{\mathbb{R}}$ is defined by
\begin{equation*}
\mbox{\rm epi}\,h:=\big\{(x,\gamma )\in X\times
\mathbb{R}\big|\;x\in \mbox{\rm dom}\,h,\ h(x)\leq \gamma \big \}.
\end{equation*}
It is easy to see that the convex counterpart of the set $C_{
\mathcal{J}}\left( p\right)$ in (\ref{C linear j}) is {\normalsize
\begin{align}
C\left( p\right) & :=\mbox{\rm co}\,\left\{ \left( u^{\ast
},f_{j}^{\ast }\left( u^{\ast }\right) +p_{j}\right) \mid j\in
J,~u^{\ast }\in \mathrm{dom}
f_{j}^{\ast }\right\}  \notag \\
& =\mbox{\rm co}\,\big(\bigcup\nolimits_{j\in J}\mbox{\rm
gph}\,(f_{j}-p_{j})^{\ast }\big)\subset X^{\ast }\times \mathbb{R}.
\label{Cconv}
\end{align}
For more details the reader is addressed to \cite{DGL07} and
particularly to the extended Farkas' Lemma, which may be found in
\cite[Theorem~4.1]{DGL07}.}

{\normalsize In this convex setting the SSC at $\sigma \left(
0\right)$ reads as $\sup_{t\in T}f_{t}(\widehat{x})<0$ for some
$\widehat{x}\in X.$ Note that $\widehat{x}$ is a strong Slater
point for $\sigma \left( 0\right) $ if and only if the same
happens for the linearized system $\sigma_{\mathcal{J}}\left(
0\right)$, i.e., $\sup_{\left( j,u^{\ast }\right) \in
T}\{\left\langle u^{\ast },\widehat{x}\right\rangle -f_{j}^{\ast
}\left( u^{\ast }\right)\}<0$.}

{\normalsize The next result, which follows from its linear
counterpart in Proposition~\ref{pr coderivJ}, computes the
coderivative of the solution map \eqref{eq_fsm} to the original
infinite convex system \eqref{eq_conv_sys} in terms of its initial
data.}

\begin{proposition}{\bf (computing coderivatives for convex systems).}
{\normalsize\label{Prop_charact_coderiv} Consider
$\overline{x}\in\mathcal{F} \left( 0\right) $ for the solution map
\eqref{eq_fsm} to the convex system \eqref{eq_conv_sys}. Then we
have $p^{\ast }\in D^{\ast }\mathcal{F}\left(0,
\overline{x}\right)\left( x^{\ast }\right)$ if and only if
\begin{equation}
\big(p^{\ast },-x^{\ast },-\left\langle x^{\ast
},\overline{x}\right\rangle \big)\in \mathrm{cl}^{\ast
}\mathrm{cone}\big(\bigcup\nolimits_{j\in J}\big[ \left\{-\delta
_{j}\right\} \times \mathrm{gph}\,f_{j}^{\ast }\big]\big).
\label{eq_charact_coderiv}
\end{equation}}
\end{proposition}

{\normalsize The next major result of the paper provides a precise
computation of the exact Lipschitzian bound of the solution map
\eqref{eq_fsm} in the case when either the set $\bigcup_{j\in J}
\mathrm{dom}\,f_{j}^{\ast }$ is bounded in $X^{\ast }$ (this is
the convex counterpart of the boundedness of $\left\{ a_{t}^{\ast
}\;|\;t\in T\right\}$) or the decision Banach space }$X$
{\normalsize is reflexive. Before this we show that the
boundedness assumption, which looks quite natural in the linear
setting, may fail in very simple convex examples.}

\begin{example} {\bf (failure of the bounded ness assumption for
convex systems).} {\rm Consider the following single inequality
involving one-dimensional decision and parameter variables:
\begin{equation}
x^{2}\leq p\;\mbox{ for }\;x,p\in \mathbb{R}.\label{ex}
\end{equation}
Note that the linearized system associated with \eqref{ex} reads
as follows:
\begin{equation*}
\left\{ ux\leq \frac{u^{2}}{4}+p,\quad u\in \mathbb{R}\right\} ,
\end{equation*}
and thus the coefficient boundedness assumption fails.}
\end{example}

\begin{theorem}{\bf (evaluation of the coderivative norm for
convex systems).} {\normalsize \label{Th_lip} For any
$\overline{x}\in\mathcal{F}\left(0\right)$ } {\normalsize we have
the relationships
\begin{align*}
& \sup\Big\{\left\Vert u^{\ast }\right\Vert
^{-1}\Big|\;\big(u^{\ast },\left\langle u^{\ast
},\overline{x}\right\rangle\big)\in \mbox{\rm cl} \,^{\ast
}\mbox{\rm co}\,\big(\bigcup\nolimits_{j\in J}\mbox{\rm
gph}\,f_{j}{}^{\ast }\big)\Big\} \\
&\le \left\Vert D^{\ast }\mathcal{F}\left( 0,\overline{x}\right)
\right\Vert \le\mathrm{lip}\,\mathcal{F}\left(
0,\overline{x}\right) .
\end{align*}}
{\normalsize If furthermore either the set $\bigcup_{j\in
J}\mathrm{dom} \,f_{j}^{\ast }$\ is bounded in $X^{\ast }$ or the
space }$X$ {\normalsize is reflexive, then the above inequalities
hold as equalities.}
\end{theorem}

\begin{proof} It follows from Theorem~\ref{Th_norm lip j} applied to the linearized system
with block perturbations by the linearization procedure and
discussions above.
\end{proof}

\begin{remark}
{\normalsize \label{ioffe} }\emph{After the publication of
\cite{CLMP09}, Alex Ioffe drew our attention to the possible
connections of some of the results therein with those obtained in
\cite{is} for general set-valued mappings of convex graph.
Examining this approach, we were able to check, in particular,
that the result of \cite[Corollary~4.7]{CLMP09} on the computing
the exact Lipschitzian bound of linear infinite systems via the
coderivative norm under the coefficient boundedness can be
obtained by applying Theorem~3 and Proposition~5 from \cite{is} by
involving some technicalities.}
\end{remark}

\begin{remark}
{\normalsize \label{ioffe 2} \emph{The main results of this paper
were basically obtained at the end of 2008 during the visit of the
third author to the University of Alicante and the Miguel
Hern\'{a}ndez University of Elche and then were presented at
several meetings in 2009-10 and also written in \cite{CLMP10}.
During the final revision of the manuscript we have become
familiar with the very recent preprint \cite{Ioffe2010} where,
under a certain uniform boundedness condition held by replacing
our functions}$ f_{j}$} \emph{with }$\max \{-1,f_{j}\},$ \emph{
the equality in Theorem~\ref{Th_lip} is obtained with no
coefficient boundedness or reflexivity assumptions by a completely
different approach.}
\end{remark}

\begin{remark}
{\rm Following our approach in \cite{CLMP09a}, the coderivative
calculations presented above allow us to develop necessary
optimality conditions of both lower subdifferential and upper
subdifferential types for nonsmooth problems of semi-infinite and
infinite programming with feasible sets given by infinite systems
of convex inequalities; see \cite[Section~6]{CLMP10} for more
details.}
\end{remark}

\end{document}